\documentclass[twocolumn]{autart}
\usepackage{bm}
\usepackage{graphics}
\usepackage{epsfig}
\usepackage{times}
\usepackage{amsmath}
\usepackage{amssymb}
\usepackage{graphicx}
\usepackage{epstopdf}
\usepackage{epsfig}
\usepackage{enumerate}
\usepackage[noadjust]{cite}
\usepackage{color}
\usepackage{makecell}
\usepackage{arydshln}
\usepackage{stmaryrd}
\usepackage{stackengine}
\usepackage{threeparttable}
\usepackage{algorithm}
\usepackage{algorithmicx}
\usepackage{xpatch}
\usepackage{bbm}
\usepackage{bbding}
\usepackage{enumerate}
\usepackage{algpseudocode}
\usepackage{multirow}
\usepackage{epsfig}
\usepackage{tikz,harvard,fancybox}
\usepackage{cuted}
\usepackage{units}
\usepackage{subfigure}

\makeatletter
\xpatchcmd\ALG@step{\arabic{ALG@line}}{\fmtlinenumber{ALG@line}}{}{}
\makeatother
\let\fmtlinenumber\arabic 

\newtheorem{Theorem}{Theorem}
\newtheorem{Assumption}{Assumption}

\newtheorem{Lemma}{Lemma}

\newfont{\bb}{msbm10}

\newcommand{\tr}{^{\sf T}}

\newcommand{\m}[1]{{\mathrm{#1}}}

\allowdisplaybreaks[4]

\begin{document}
\begin{frontmatter}

\title{Local Adapt-Then-Combine Algorithms for
Distributed Nonsmooth Optimization: Achieving
Provable Communication Acceleration\thanksref{footnoteinfo1}} 

\thanks[footnoteinfo1]{This work was supported in part by the National Key R\&D Project of China  under Grant Nos. 2020YFA0714300, and the National Natural Science Foundation of China under Grant Nos. 62576098 and 62473098.}

\thanks[footnoteinfo2]{Corresponding author: Jinde Cao.}

\author[Comp,Math]{Luyao Guo}\ead{lyguo@szut.edu.cn},    
\author[Cyber]{Xinli Shi}\ead{xinli\_shi@seu.edu.cn},
\author[Math]{Wenying Xu}\ead{wyxu@seu.edu.cn},
\author[Math]{Jinde Cao\thanksref{footnoteinfo2}}\ead{jdcao@seu.edu.cn}

\address[Comp]{School of Computer Science and Engineering, Suzhou University of Technology, Suzhou 215500, China}
\address[Math]{School of Mathematics, Southeast University, Nanjing 210096, China}
\address[Cyber]{School of Cyber Science and Engineering, Southeast University, Nanjing 210096, China}


\begin{abstract}                          
This paper is concerned with the distributed composite optimization problem over networks, where agents aim to minimize a sum of local smooth components and a common nonsmooth term. Leveraging the probabilistic local updates mechanism, we propose a communication-efficient Adapt-Then-Combine (ATC) framework, FlexATC, unifying numerous ATC-based distributed algorithms. Under stepsizes independent of the network topology and the number of local updates, we establish sublinear and linear convergence rates for FlexATC in convex and strongly convex settings, respectively. Remarkably, in the  strong convex setting, the linear rate is decoupled from the objective functions and network topology, and FlexATC permits communication to be skipped in most iterations without any deterioration of the linear rate. In addition, the proposed unified theory demonstrates for the first time that local updates provably lead to communication acceleration for ATC-based distributed algorithms. Numerical experiments further validate the efficacy of the proposed framework and corroborate the theoretical results.
\end{abstract}

\end{frontmatter}

\section{Introduction}
Consider the following composite distributed optimization problem involving a set of agents $[n]\triangleq\{1, 2, \dots, n\}$, where each agents has a private loss function $f_i:\mathbb{R}^d\rightarrow \mathbb{R}$ ($L$-smooth) and public loss function $r:\mathbb{R}^d\rightarrow \mathbb{R}\cup\{\infty\}$ (not required to be smooth)
\begin{equation}\label{EQ:Problem1}
\begin{aligned}
\min_{x\in \mathbb{R}^d}\left\{\frac{1}{n}\sum_{i=1}^{n}\Big[f_i(x)+r(x)\Big]\right\}.
\end{aligned}
\end{equation}
In this context, a network of agents collaborates to minimize the average of the global objective. In recent years, there has been growing interest in addressing the optimization problem through distributed approaches. The main driving force behind these methods is the potential of decentralization to remove the necessity for data sharing and centralized synchronization, thereby reducing the high latency typically associated with centralized computing architectures \cite{Nedic20151,Nedic2018}.

A direct approach to tackle distributed optimization problems is to utilize the gradient descent method in combination with the underlying network structure \cite{Nedic2009,Yuan2016}. However, in the presence of heterogeneous data, these primal algorithms are vulnerable to the phenomenon known as ``client-drift." This occurs because the differing local functions at each node cause each client to converge to the minima of its respective function $f_i(x)+r(x)$, which may be significantly different from the solution of \eqref{EQ:Problem1}. On the other hand, recent advancements in the primal-dual framework have introduced alternative techniques to address the challenges posed by heterogeneous data in distributed optimization, as demonstrated in \cite{shi2018augmented,liang2019exponential,EXTRA,PGEXTRA,Harnessing,DIGing,Guo2022,Guo2023,NIDS,ExactDiffusion,D2,Yuan2021,AugDGM1,AugDGM2,SONATA,Sulaiman2021,Xu2021,song2024optimal}. Despite these advances, distributed optimization algorithms still encounter challenges related to communication bottlenecks.

\begin{table*}[!t]
\renewcommand\arraystretch{2}
\begin{center}
\caption{Comparison with existing ATC-based algorithm frameworks. LU = Local Updates, Comm. = Communication, Acc. = Acceleration.}
\scalebox{0.88}{
\begin{tabular}{ccccccc}
\hline
\multirow{2}{*}{Algorithm} &\# of Comm.: $p=1$&LU?&Comm. Acc.?&\# of Comm.: $p\in(0,1]$&LU?&Comm. Acc.?\\
&\multicolumn{3}{c}{literature \cite{Xu2021,Sulaiman2021}} &\textbf{this paper}&\textbf{this paper}&\textbf{this paper}\\
\hline
(Prox-)NIDS/ED/D2&$\mathcal{O}\left(\frac{\kappa}{1-\rho}\log \epsilon^{-1}\right)$&\textcolor[rgb]{0.7,0,0}{\XSolidBrush}&\textcolor[rgb]{0.7,0,0}{\XSolidBrush}
&$\mathcal{O}\left(\frac{\kappa}{1-\rho}\log \epsilon^{-1}\right)
\xrightarrow[\text{\textbf{improved}}]{p=\frac{1}{\sqrt{\kappa(1-\rho)}}}
\mathcal{O}\left(\sqrt{\frac{\kappa}{1-\rho}}\log\epsilon^{-1}\right)$&\textcolor[rgb]{0,0.6,0}{\Checkmark}&\textcolor[rgb]{0,0.6,0}{\Checkmark}\\
(Prox-)MG-ED&$\mathcal{O}\left(\frac{\kappa}{\sqrt{1-\rho}}\log\epsilon^{-1}\right)$&\textcolor[rgb]{0.7,0,0}{\XSolidBrush}&\textcolor[rgb]{0.7,0,0}{\XSolidBrush}
&$\mathcal{O}\left(\frac{\kappa}{\sqrt{1-\rho}}\log\epsilon^{-1}\right)
\xrightarrow[\text{\textbf{improved}}]{p=\frac{1}{\sqrt{\kappa}}}
\mathcal{O}\left(\sqrt{\frac{\kappa}{1-\rho}}\log\epsilon^{-1}\right)$&\textcolor[rgb]{0,0.6,0}{\Checkmark}&\textcolor[rgb]{0,0.6,0}{\Checkmark}\\
(Prox-)ATC-GT&$\mathcal{O}\left(\frac{\kappa}{(1-\rho)^2}\log\epsilon^{-1}\right)$&\textcolor[rgb]{0.7,0,0}{\XSolidBrush}&\textcolor[rgb]{0.7,0,0}{\XSolidBrush}&
$
\mathcal{O}\left(\frac{\kappa}{(1-\rho)^2}\log\epsilon^{-1}\right)
\xrightarrow[\text{\textbf{improved}}]{p=\frac{1}{\sqrt{\kappa(1-\rho)^2}}}
\mathcal{O}\left(\frac{\sqrt{\kappa}}{1-\rho}\log\epsilon^{-1}\right)
$&\textcolor[rgb]{0,0.6,0}{\Checkmark}&\textcolor[rgb]{0,0.6,0}{\Checkmark}\\
MG-SONATA&$\mathcal{O}\left(\frac{\kappa}{\sqrt{1-\rho}}\log\epsilon^{-1}\right)$&\textcolor[rgb]{0.7,0,0}{\XSolidBrush}&\textcolor[rgb]{0.7,0,0}{\XSolidBrush}&\
$
\mathcal{O}\left(\frac{\kappa}{\sqrt{1-\rho}}\log\epsilon^{-1}\right)
\xrightarrow[\text{\textbf{improved}}]{p=\frac{1}{\sqrt{\kappa}}}
\mathcal{O}\left(\sqrt{\frac{\kappa}{1-\rho}}\log\epsilon^{-1}\right)
$&\textcolor[rgb]{0,0.6,0}{\Checkmark}&\textcolor[rgb]{0,0.6,0}{\Checkmark}\\
\hline
\end{tabular}}
\label{Table-complexity}
\end{center}
\end{table*}

To reduce communication costs, there has been a growing interest in \emph{local updates} methods, where multiple iterations of local computations are performed between communication steps \cite{Alghunaim2023,Nguyen2022,Liu2023,ProxSkip,MG-Skip}. These approaches aim to strike a balance between computation and communication, allowing agents to make progress on their local objectives without frequent exchanges of information, thus alleviating the communication bottleneck. In \cite{Alghunaim2023}, leveraging the ED \cite{ExactDiffusion,Yuan2021} (or NIDS \cite{NIDS} or D2 \cite{D2}) algorithm and employing deterministic periodic local updates, the local-ED (LED) algorithm was introduced. Similarly, in \cite{Nguyen2022} and \cite{Liu2023}, by enabling communication-efficient local updates in gradient tracking (GT), two local-GT algorithms were proposed. Although LED and local-GT have been shown to accelerate communication through local stochastic gradient steps, it has yet to be demonstrated that these local deterministic gradient steps can also reduce communication complexity. In addition, the stepsize selection in both LED and local-GT is inversely related to the number of local gradient steps. It implies that an increase in the number of local updates results in smaller permissible stepsizes, which in turn leads to slower convergence. To accelerate communication in deterministic settings, based on ED/NIDS/D2, ProxSkip \cite{ProxSkip,RandProx} has been introduced, which incorporates probabilistic local updates, meaning that communication occurs with a certain probability. This advancement is achieved without relying on traditional acceleration techniques, and the stepsize for ProxSkip is of the order $\mathcal{O}\left(\nicefrac{1}{L}\right)$, which is independent of the number of local updates. However, this approach is effective only for smooth loss function and the network exhibits sufficient connectivity. However, this approach is effective only when the network has sufficient connectivity and the loss function is smooth. To address this limitation, in \cite{MG-Skip}, the combination of communication skipping methods and multiple-round gossip communication has led to the development of MG-Skip.

In this paper, using the probabilistic local updates method, we provide a \underline{Flex}ible \underline{A}dapt-\underline{T}hen-\underline{C}ombine (ATC)-based distributed algorithm (FlexATC) for the composite distributed optimization problem \eqref{EQ:Problem1}. Table \ref{Table-complexity} shows the comparison with existing ATC-based algorithm frameworks. The key contributions are summarized as follows.
\begin{itemize}
  \item We introduce a novel unified distributed optimization algorithm framework, FlexATC, with network-independent and local update count-independent stepsizes, which is advantageous for practical implementation.
      FlexATC integrates numerous existing ATC-based distributed algorithms, encompassing both classical algorithms \cite{NIDS,ExactDiffusion,Yuan2021,AugDGM1,AugDGM2,SONATA} and locally updated algorithms \cite{Alghunaim2023,Nguyen2022,ProxSkip,MG-Skip}. To the best of our knowledge, FlexATC is the first unified ATC-based algorithm framework to incorporate probabilistic local updates.
  \item In the convex setting, we establish the $\mathcal{O}\left(\nicefrac{1}{K}\right)$ convergence rate of FlexATC. In the strongly convex setting, we establish the linear convergence rate of FlexATC. Importantly, we show that the derived rate is decoupled from the local loss function and network topology, matching the linear rate of the centralized proximal. Furthermore, we demonstrate that FlexATC can skip communications with a probability $p$ during execution without affecting the linear convergence rate.
  \item Previous researches \cite{ProxSkip,RandProx,MG-Skip} have exclusively showcased the benefits of probabilistic local updates for ED/NIDS/D2. However, whether other algorithms can also benefit from this mechanism remains an open question. In this paper, we provide the first evidence that probabilistic local updates provably accelerate communication for ATC-based distributed algorithms.
\end{itemize}

The rest of the paper is organized as follows. Section \ref{SEC3} presents FlexATC and explores its connections to existing approaches.
Section \ref{SEC4} provides the convergence and complexity analysis, and demonstrate that probabilistic local updates effectively accelerate communication in ATC-based algorithms.
Section \ref{SEC5} includes several numerical simulations to corroborate the theoretical results. Section \ref{SEC6} concludes the paper.

\emph{Notations:} Let $\mathbb{R}^n$ denote the $n$-dimensional vector space equipped with the inner product $\langle \cdot, \cdot \rangle$. The norms considered in this paper are the $\ell_1$-norm and $\ell_2$-norm, denoted as $\|\cdot\|_1$ and $\|\cdot\|$, respectively. The zero matrix and identity matrix of appropriate dimensions are represented $\bm{0}$ and $\bm{I}$, respectively. The vector with all components equal to one is represented as $\bm{1}$. For a matrix $H \in \mathbb{R}^{p \times n}$, we define the operator norm as $\|H\| \triangleq \max_{\|s\| = 1} \|Hs\|$, and $\sigma_m(H)$ refers to the minimum nonzero singular value of $H$. The notation $H \succeq 0$ indicates that $H$ is positive semi-definite, $\mathrm{null}(\cdot)$ and $\mathrm{span}(\cdot)$ represent the null space and the range space of a vector or matrix, respectively. The probability of event $s$ occurring is given by $\mathrm{Prob}(s)$. We use $\m{col}\{a_1, \ldots, a_n\}$ to represent the column vector formed by stacking $a_1, \ldots, a_n$ vertically, i.e., $[a_1\tr, \ldots, a_n\tr]\tr$, and $\lfloor \cdot \rfloor$ denotes rounding toward $-\infty$. Finally, the Kronecker product is denoted by $\otimes$.

\section{Framework: FlexATC}\label{SEC3}
\subsection{Problem and network settings}
\begin{Assumption}\label{ASS1}
$f_i$ is $\mu$-strongly convex with $\mu\geq0$; $f_i$ is $L$-smooth with $L>0$, i.e,
$\|\nabla f_i(x^{(1)})-\nabla f_i(x^{(2)})\|\leq L \|x^{(1)}-x^{(2)}\|, \text{ for any } x^{(1)},x^{(2)}\in \mathbb{R}^d$.
$r$ proper, closed, and convex; $r$ is proximable, i.e., for $\alpha > 0$, its proximal mapping is defined as
  $\mathrm{prox}_{\alpha r}(x)\triangleq\mathrm{argmin}_{s\in\mathbb{R}^d}\{r(s)+\nicefrac{1}{2\alpha}\|s-x\|^2\}$,
and has an analytical solution or can be calculated efficiently.
\end{Assumption}

Consider a network represented by an undirected and connected graph $\mathcal{G}(\mathcal{N}, \mathcal{E})$, where $\mathcal{N}\triangleq\{1,2,\cdots,n\}$ is the set of vertices, and $\mathcal{E} \subseteq \mathcal{N} \times \mathcal{N}$ is the set of edges. An edge between agents $i$ and $j$ exists if and only if $(i,j) \in \mathcal{E}$, meaning there is a communication link between these two agents. Define $\mathcal{N}_i\triangleq\{j\in\mathcal{N}:(i,j)\in\mathcal{E}\}$. Let the global mixing matrix $W = [W_{ij}]_{i,j=1}^n\in\mathbb{R}^{n\times n}$ and impose the following assumption.
\begin{Assumption}\label{MixingMatrix-1}
$W\tr=W$, $W_{ij} = W_{ji} = 0$ if $(i,j) \notin \mathcal{E}$, and $W_{ij}=W_{ji} > 0$ if $(i,j) \in \mathcal{E}$; $W\bm{1}=\bm{1}$.
\end{Assumption}
Under Assumption \ref{MixingMatrix-1}, if the eigenvalues of $W$ are denoted as $\lambda_1 \geq \lambda_2 \geq \lambda_3 \geq \cdots \geq \lambda_n$, then the following relationships hold: $1 = \lambda_1 > \lambda_2 \geq \cdots \geq \lambda_n > -1$. The spectral gap is defined as $\rho \triangleq \max\{|\lambda_2|, |\lambda_n|\} \in (0,1)$. Consider $\bm{W} \triangleq W \otimes \bm{I}_d \in \mathbb{R}^{nd \times nd}$. Let $\bm{v} \triangleq \text{col}\{v_1, \ldots, v_n\}$, where each $v_i \in \mathbb{R}^d$ represents the local information of agent $i$. The $i$-th block of the vector $\bm{Wv}$ is expressed as $\text{col}\{\sum_{j \in \mathcal{N}_i} W_{ij} v_j\}$, and this can be computed by agent $i$ via localized communication with its neighbors. This operation corresponds to one communication round. Consequently, $\bm{W}^N \bm{v}$, where $N > 1$ represents $N$ consecutive rounds of communication.

\subsection{Algorithm description}
Let $\bm{x}^k \triangleq \text{col}\{x_1^k, x_2^k, \ldots, x_n^k\}$ represent the global state at the $k$-th iteration, where $x_i^k \in \mathbb{R}^d$ denotes the local state of node $i$ at the $k$-th iteration. To describe the algorithm, we also define the following network quantities: $F(\bm{x})\triangleq\sum_{i=1}^{n}f_i(x_i)$, $R(\bm{x})\triangleq\sum_{i=1}^{n}r(x_i)$. We further introduce matrices $\bm{A}=A\otimes \bm{I}_d$, $\bm{B}=B\otimes \bm{I}_d$ which are polynomial functions of $\bm{W}$.
\begin{Assumption}\label{MixingMatrix-2}
$A\tr=A$ and $A\bm{1}=\bm{1}$, $B\succeq\bm{0}$ and satisfies $\mathrm{null}(B)=\mathrm{span}(\bm{1})$, and $\bm{I}-A^2-B\succeq\bm{0}$.
\end{Assumption}
In addition, let $\bm{A}_k\triangleq \theta_k \bm{A} + (1-\theta_k)\bm{I}$ and $\bm{B}_k \triangleq \theta_k \bm{B}$,
where $\theta_k\in\{0,1\}$ and $\mathop{\m{Prob}}(\theta_k =1) = p$. With these preparations, we develop FlexATC as the following form:
\begin{subequations}\label{Framework-UP}
\begin{align}
\bm{w}^k &= \bm{x}^k -\alpha \nabla F(\bm{x}^k), \label{Framework-UP-step1}\\
\bm{x}^{k+1} & = \mathrm{prox}_{\alpha R}\left(\bm{A}_k\left(\bm{w}^k+\bm{y}^k\right)\right), \label{Framework-UP-step2}\\
\bm{y}^{k+1} & = \bm{y}^k - p\bm{B}_k\left(\bm{w}^k+\bm{y}^k\right), \label{Framework-UP-step3}
\end{align}
\end{subequations}
where $\alpha>0$ is the stepsize. In the implementation of FlexATC, step \eqref{Framework-UP-step1} is the adaptation step, and steps \eqref{Framework-UP-step2} and \eqref{Framework-UP-step3} are combination steps. It can be observed from FlexATC \eqref{Framework-UP} that the communication is not required in every iteration and is instead triggered with a certain probability $p$. Thus, the number of local gradient steps of FlexATC is random, and from the perspective of expectation, the communication occurs every $\nicefrac{1}{p}$ iterations. Different from deterministic periodic local updates, where the frequency of local gradient steps is fixed within a determined period, we adopt a probabilistic local update mechanism to reduce communication frequency. This approach offers more flexibility in the frequency of local updates. Specifically, before the algorithm begins, the communication rounds can be randomly determined using a coin-flip method. This process can be executed on any single node and then broadcast to the other nodes, requiring only one round of communication. Since $\theta_k$ is encoded in \( \{0, 1\} \), the additional communication overhead is negligible. Moreover, this process does not result in any leakage of privacy information. Algorithm \ref{alg-FlexATC} elaborates on the distributed implementation of FlexATC.

\begin{algorithm}[!t]
  \caption{FlexATC} 
    \label{alg-FlexATC}
  \begin{algorithmic}[1]
    \Require
      stepsize $\alpha\in(0,\nicefrac{2}{L})$; probability $p\in(0,1]$; matrices $A=[A_{ij}]_{i,j=1}^n$, $B=[B_{ij}]_{i,j=1}^n$; sequence of independent coin flips $\{\theta_0$, $\theta_1$, $\ldots$ $\theta_{K-1}\}$ with $\mathop{\m{Prob}}(\theta_k =1) = p$; initial $x^0_{i}\in\mathbb{R}^d$ and $y^0_i = 0\in\mathbb{R}^d,i=1,\ldots,n$
    \For{$k=0,1,2,\ldots,K-1$}
      \State Adaptation: $w^k_i=x^k_i-\alpha\nabla f_i(x^k_i)$
       \If {$\theta_k=1$}
            \State  Combination:
                $\bigg\{\begin{array}{l}
                a_i^k=\sum_{j\in \mathcal{N}_i}A_{ij}(w^k_j+y^k_j) \\
                b_i^k=\sum_{j\in \mathcal{N}_i}B_{ij}(w^k_j+y^k_j)
                \end{array}$
            \State Correction: $x_i^{k+1}=\m{prox}_{\alpha r}(a_i^k)$, $y_i^{k+1}=y_i^k-p b_i^k$

        \Else
            \State Skip Communication:
             \State  $x^{k+1}_i=\mathrm{prox}_{\alpha r}(w^k_i+y_i^k)$,
                $y^{k+1}_i =  y_i^k$
        \EndIf
    \EndFor
    \Ensure
      $\bm{x}^K$.
  \end{algorithmic}
\end{algorithm}

\subsection{Relation to existing ATC algorithms}\label{Relation-1}
In this subsection, we delve into the associations between FlexATC and existing ATC-based distributed algorithms (refer to Table \ref{Table-relation}). We demonstrate that by making specific selections for $\bm{A}$ and $\bm{B}$, several ATC-based distributed algorithms can be derived from the proposed FlexATC \eqref{Framework-UP}.

We initiate a discussion on the correlation between FlexATC, with $p=1$ and $r=0$ as special scenario, and ED/NIDS/D2 \cite{ExactDiffusion,NIDS,D2}, MG-ED \cite{Yuan2021}, and ATC-GT (encompassing AugDGM \cite{AugDGM1}, ATC-DIGing \cite{AugDGM2}, and SONATA \cite{SONATA}). To this case, letting $\bm{z}^k\triangleq\bm{w}^k+\bm{y}^k$ and $\triangle_F(\bm{x}^k)\triangleq\nabla F(\bm{x}^{k})-\nabla F(\bm{x}^{k-1})$, it follows from \eqref{Framework-UP-step1} and \eqref{Framework-UP-step3} that
$$\bm{z}^k=\bm{z}^{k-1}-\bm{x}^{k-1}-\bm{B}\bm{z}^{k-1}+\bm{x}^k-\alpha\triangle_F(\bm{x}^k).$$ Since $\bm{A}$ and $\bm{B}$ are polynomial functions of $\bm{W}$, we have $\bm{AB}=\bm{BA}$. Thus, it holds that
\begin{align*}
\bm{A}\bm{z}^k=\bm{A}(\bm{z}^{k-1}-\bm{x}^{k-1}-\bm{B}\bm{z}^{k-1})+\bm{A}(\bm{x}^k-\alpha\triangle_F(\bm{x}^k)).
\end{align*}
Combining with $\bm{x}^{k+1}=\bm{A}\bm{z}^k$, FlexATC ($p=1$ and $r=0$) becomes the following scheme
\begin{align}\label{FlexATC-primal}
\bm{x}^{k+1}&=\bm{x}^{k}-\bm{A}\bm{x}^{k-1}-\bm{B}\bm{x}^{k}+\bm{A}(\bm{x}^k-\alpha\triangle_F(\bm{x}^k)).
\end{align}

\noindent \textbf{1) ED/NIDS/D2 and MG-ED:} Letting $\bm{A}+\bm{B}=\bm{I}$, the update of FlexATC will become
\begin{align*}
\bm{x}^{k+1}&=\bm{A}(\bm{x}^{k}-\bm{x}^{k-1})+\bm{A}(\bm{x}^{k} -\alpha\triangle_F(\bm{x}^k)).
\end{align*}
Therefore, by selecting $(\bm{A},\bm{B})=(\bm{I}-c(\bm{I}-\bm{W}),c(\bm{I}-\bm{W}))$, where $c\in(0,\nicefrac{1}{2}]$, NIDS and D2 can be derived from FlexATC ($p=1$ and $r=0$). Similarly, by choosing
$(\bm{A},\bm{B})=(\nicefrac{1}{2}(\bm{I}+\bm{W}),\nicefrac{1}{2}(\bm{I}-\bm{W}))$, ED can be derived from FlexATC. Furthermore, under the conditions of $\bm{W}\succeq\bm{0}$, choosing $(\bm{A},\bm{B})=(\bm{W}^N,\bm{I}-\bm{W}^N)$, where $N>1$, MG-ED can be also derived from FlexATC.

\noindent \textbf{2) ATC-GT:} Recall the update of ATC-GT
\begin{align*}
\bm{x}^{k+1}=\bm{W}\left(\bm{x}^k-\alpha \bm{y}^k\right),\
\bm{y}^{k+1}=\bm{W}\left(\bm{y}^{k}+\triangle_F(\bm{x}^{k+1})\right).
\end{align*}
From the first step, we have $\bm{x}^{k+1}-\bm{W}\bm{x}^k=-\alpha\bm{W}\bm{y}^k$. By multiplying both sides of the second step by $-\alpha\bm{W}$, we have
$$
-\alpha\bm{W}\bm{y}^{k+1}=-\alpha\bm{W}^2\bm{y}^{k}-\alpha\bm{W}^2\triangle_F(\bm{x}^{k+1}).
$$
It implies that
\begin{align*}
\bm{x}^{k+2}-\bm{W}\bm{x}^{k+1}&=\bm{W}\bm{x}^{k+1}-\bm{W}^2\bm{x}^{k}-\alpha\bm{W}^2\triangle_F(\bm{x}^{k+1}).
\end{align*}
Then, the update of ATC-GT is equivalent to
$$
\bm{x}^{k+1}=2\bm{W}\bm{x}^k-\bm{W}^2\left(\bm{x}^{k-1}-\alpha\triangle_F(\bm{x}^k)\right),
$$
which is an instance of the general scheme FlexATC ($p=1$ and $r=0$) \eqref{FlexATC-primal}, with $\bm{W}\succeq\bm{0}$ and $(\bm{A},\bm{B})=(\bm{W}^2,(\bm{I}-\bm{W})^2)$. Moreover, with $\bm{W}\succeq\bm{0}$ and $(\bm{A},\bm{B})=(\bm{W}^{2N},(\bm{I}-\bm{W}^N)^2)$, FlexATC \eqref{FlexATC-primal} is equivalent to MG-SONATA.

Then, we consider the case where $r\neq0$ and $p=1$, it follows from \eqref{Framework-UP-step1} and \eqref{Framework-UP-step3} that FlexATC becomes
\begin{align*}
\bm{z}^k&=\bm{z}^{k-1}-\bm{x}^{k-1}-\bm{B}\bm{z}^{k-1}+\bm{x}^k-\alpha\triangle_F(\bm{x}^k),\\
\bm{x}^{k+1} & = \mathrm{prox}_{\alpha R}(\bm{A}\bm{z}^k).
\end{align*}
By comparison, one can see that it is equivalent to ATC-PUDA \cite{Sulaiman2021}. Let $\bm{s}^k=\bm{A}\bm{z}^k$. Since $\bm{AB}=\bm{BA}$, FlexATC ($p=1$ and $r=0$) can also be written as
\begin{align*}
\bm{s}^k&=(\bm{I}-\bm{B}) \bm{s}^{k-1}-\bm{A}\bm{x}^{k-1}+\bm{A}(\bm{x}^k-\alpha\triangle_F(\bm{x}^k)),\\
\bm{x}^{k+1} & = \mathrm{prox}_{\alpha R}(\bm{s}^k),
\end{align*}
which is equivalent to ATC-ABC \cite{Xu2021}.

\begin{table}[!t]
\renewcommand\arraystretch{1.6}
\begin{center}
\caption{The relation of FlexATC to existing ATC-based distributed algorithms for specific choices of $\bm{A}$ and $\bm{B}$.}
\scalebox{0.8}{
\begin{tabular}{ccccccc}
\hline
Algorithm & $r\neq0$ &  $\bm{A}$ & $\bm{B}$  & \# of Comm.\\
\hline
(Prox-)ED/NIDS/D2 &\Checkmark& $\bm{I}-c(\bm{I}-\bm{W})$&$c(\bm{I}-\bm{W})$&$1$\\
(Prox-)MG-ED &\Checkmark&  $\nicefrac{1}{2}(\bm{I}+\bm{W}^N)$&$\nicefrac{1}{2}(\bm{I}-\bm{W}^N)$&$N$\\
(Prox-)ATC-GT&\Checkmark& $\bm{W}^2$ & $(\bm{I}-\bm{W})^2$  &$2$\\
MG-SONATA &\Checkmark& $(\bm{W}^N)^2$ & $(\bm{I}-\bm{W}^N)^2$  &$2N$\\
LED &-& $\bm{W}$ & $\bm{I}-\bm{W}$ & $p$ (periodic) \\
local-GT& -& $\bm{W}^2$ & $(\bm{I}-\bm{W})^2$ & $2p$ (periodic) \\
ProxSkip&-&  $\bm{I}-c(\bm{I}-\bm{W})$ & $c(\bm{I}-\bm{W})$ & $p$ (random)\\
\hline
\end{tabular}}
\label{Table-relation}
\end{center}
\end{table}

Finally, we consider the case where $p\in(0,1]$. By comparison, it is easy to verify that ProxSkip \cite{ProxSkip} is a instance of FlexATC with $(\bm{A},\bm{B})=(\bm{I}-c(\bm{I}-\bm{W}),c(\bm{I}-\bm{W}))$. In addition, to show the connection with periodic local updates based algorithms, from the perspective of expectations, we rewrite FlexATC \eqref{Framework-UP} as
\begin{subequations}\label{Framework-UP-periodic}
\begin{align}
\bm{\psi}^r_{t+1}& = \bm{\psi}_t^r -\alpha \nabla F(\bm{\psi}_t^r) + \bm{y}^r, \text{ for } t = 1,\cdots,M,\\
\bm{y}^{r+1}& = \bm{y}^{r} - \nicefrac{1}{M}\bm{B}\bm{\psi}^r_{M+1},\\
\bm{x}^{r+1}& = \m{prox}_{\alpha R}\left(\bm{A}\bm{\psi}^r_{M+1}\right),
\end{align}
\end{subequations}
where $r$ denotes the communication round, $M=\lfloor\nicefrac{1}{p}\rfloor$ denotes the number of local updates, and $\bm{\psi}^r_{1}=\bm{x}^r$. When iterating through $\bm{\psi}^r_{t+1} = \bm{\psi}_t^r -\alpha \nabla F(\bm{\psi}_t^r) + \bm{y}^r, \text{ for } t = 1,\ldots,M$, it holds that
$\bm{\psi}^r_{M+1}=\bm{x}^r-\alpha\sum_{t=1}^{M}\nabla F(\bm{\psi}_t^r)+M\bm{y}^r$.
Then, by eliminating $\bm{y}^r$, \eqref{Framework-UP-periodic} takes the following form
\begin{align*}
\bm{\psi}^r_{M+1}&=(\bm{I}-\bm{B})\bm{\psi}^{r-1}_{M+1}+(\bm{x}^r-\bm{x}^{r-1})-\alpha\sum_{t=1}^{M}\triangle_F(\bm{\psi}_t^r),\\
\bm{x}^{r+1}& = \m{prox}_{\alpha R}\left(\bm{A}\bm{\psi}^r_{M+1}\right),
\end{align*}
with $\bm{\psi}^r_{0}=\bm{x}^r$. When $r=0$, one has $\bm{x}^{r+1}=\bm{A}\bm{\psi}^r_{M+1}$ and
\begin{align*}
\bm{x}^{r+1}&=\bm{x}^{r}-\bm{A}\bm{x}^{r-1}-\bm{B}\bm{x}^{r}+\bm{A}(\bm{x}^r-\alpha\sum_{t=1}^{M}\triangle_F(\bm{\psi}_t^r)).
\end{align*}
Therefore, LED \cite{Alghunaim2023} is an instance of FlexATC with $\bm{W}\succeq\bm{0}$ and $(\bm{A},\bm{B})=(\bm{W},\bm{I}-\bm{W})$. Recall local-GT \cite{Nguyen2022}. Similar as \cite[Proposition 3.1]{Liu2023}, local-GT can be rewritten as follows.
\begin{align*}
\bm{x}^{r+1}&=\bm{W}\left(\bm{x}^r-\alpha \bm{y}^r\right),\\
\bm{y}^{r+1}&=\bm{W}\Big(\bm{y}^{r}+\sum_{t=1}^{M}(\nabla F(\bm{\psi}_t^{r+1})-\nabla F(\bm{\psi}_t^r))\Big).
\end{align*}
Thus, with $\bm{W}\succeq\bm{0}$ and $(\bm{A},\bm{B})=(\bm{W}^2,(\bm{I}-\bm{W})^2)$, local-GT is an instance of FlexATC.

\section{Main theoretical results}\label{SEC4}
This section first establishes the sublinear convergence rate for FlexATC \eqref{Framework-UP} in the convex setting, and then derives the linear convergence rate in the strongly convex setting. Finally, it proceeds to analyze the complexity of FlexATC.

We commence by introducing the ensuing lemma.

\begin{Lemma}\label{LEM-1}
Under Assumptions \ref{ASS1}, \ref{MixingMatrix-1}, and \ref{MixingMatrix-2},
if the point $(\bm{x}^{\star},\bm{w}^{\star},\bm{u}^\star)$ satisfies that
\begin{subequations}\label{KKT-Condition}
\begin{align}
\bm{w}^{\star}&=\bm{x}^{\star}-\alpha \nabla F(\bm{x}^{\star}),\label{KKT-Condition-1}\\
\bm{x}^{\star}&=\mathrm{prox}_{\alpha R}(\bm{A} (\bm{w}^{\star}-\sqrt{\bm{B}}\bm{u}^\star)), \label{KKT-Condition-2}\\
\bm{0}&=\sqrt{\bm{B}}(\bm{w}^{\star}-\sqrt{\bm{B}}\bm{u}^\star),\label{KKT-Condition-3}
\end{align}
\end{subequations}
then $\bm{x}^{\star}=\bm{1}_n\otimes x^\star$ with $x^{\star}\in \mathbb{R}^d$ solving problem \eqref{EQ:Problem1}.
\end{Lemma}

{\it Proof}: With Assumption \ref{MixingMatrix-2}, it holds that $\m{null}(\bm{B})=\mathrm{span}(\bm{1}_n\otimes\bm{I}_d)$. Then, from $\bm{0}=\bm{B}(\bm{w}^{\star}-\sqrt{\bm{B}}\bm{u}^\star)$, we have that the block elements of $\bm{w}^{\star}-\sqrt{\bm{B}}\bm{u}^\star$ are equal to each other. Since $A\bm{1}=\bm{1}$, we have that there exists $z^\star\in\mathbb{R}^d$ such that
$
\bm{A}(\bm{w}^{\star}-\sqrt{\bm{B}}\bm{u}^\star) = \bm{1}_n\otimes z^\star
$.
From \eqref{KKT-Condition-2} and the definition of the $\mathrm{prox}_{\alpha R}(\cdot)$, we have
$\bm{x}^{\star}=\mathrm{argmin}_{\bm{x}\in\mathbb{R}^{nd}}\{R(\bm{x})+\nicefrac{1}{2\alpha}\|\bm{x}-\bm{A}(\bm{w}^{\star}-\sqrt{\bm{B}}\bm{u}^\star)\|^2\}$.
It deduces that
\begin{align}\label{proof-lem-eq1}
x_i^{\star}=\mathrm{argmin}_{x\in\mathbb{R}^{d}}\{r(x)+\nicefrac{1}{2\alpha}\|x-z^\star\|^2\},\  i\in [n].
\end{align}
Thus, it holds that $x_1^\star=\cdots=x_n^\star\triangleq x^\star$. Then, we prove $x^\star$ solves problem \eqref{EQ:Problem1}. It follows from \eqref{proof-lem-eq1} that
\begin{align}\label{proof-lem-eq2}
\nicefrac{1}{\alpha}(z^\star-x^\star)\in \partial r(x^\star).
\end{align}
From $\m{null}(\sqrt{B})=\mathrm{span}(\bm{1})$, it holds that $(\bm{1}_n\otimes \bm{I}_d)\tr(\bm{w}^\star-\sqrt{\bm{B}}\bm{u}^\star)=nz^\star$.
In addition, from \eqref{KKT-Condition-1}, one has $(\bm{1}_n\otimes \bm{I}_d)\tr(\bm{w}^\star-\sqrt{\bm{B}}\bm{u}^\star)=(\bm{1}_n\otimes \bm{I}_d)\tr(\bm{x}^{\star}-\alpha \nabla F(\bm{x}^{\star})-\sqrt{\bm{B}}\bm{u}^\star)=nx^\star-\alpha\sum_{i=1}^{n}f_i(x^\star)$.
Therefore, $\nicefrac{1}{\alpha}(x^{\star}-z^{\star})-\nicefrac{1}{n}\sum_{i=1}^{n}\nabla f_i(x^{\star})=0$.
Together with \eqref{proof-lem-eq2}, we have $0\in \nicefrac{1}{n}\sum_{i=1}^{n}[\nabla f_i(x^{\star})+\partial r(x^\star)]$.
Thus, $x^\star$ solves problem \eqref{EQ:Problem1}.
\hfill{$\square$}

Let $\bm{u}^k$ satisfy $\bm{y}^k = -\sqrt{\bm{B}}\bm{u}k$ and $\bm{u}^0=\bm{0}$. Since $\bm{y}^0=\bm{u}^0=\bm{0}$, the update \eqref{Framework-UP} can be equivalently written as follows, since they generate an identical sequence $\{\bm{x}^k\}_{k=0}$.
\begin{subequations}\label{Framework-UP-2}
\begin{align}
  \bm{w}^k &= \bm{x}^k -\alpha \nabla F(\bm{x}^k), \\
  \bm{x}^{k+1} & = \mathrm{prox}_{\alpha R}(\bm{A}_k(\bm{w}^k-\sqrt{\bm{B}}\bm{u}^k)), \\
  \bm{u}^{k+1} & = \bm{u}^k + p \sqrt{\bm{B}_k}(\bm{w}^k-\sqrt{\bm{B}}\bm{u}^k), \label{Framework-UP-2-3}
\end{align}
\end{subequations}
where $\sqrt{\bm{B}_k}=\theta_k\sqrt{\bm{B}}$. Let the point $(\bm{x}^{\star},\bm{w}^{\star},\bm{u}_b^\star)$ satisfies the optimal condition \eqref{KKT-Condition} and $\bm{u}_b^\star\in\m{span}(\sqrt{\bm{B}})$. We then present the following lemma, which plays a crucial role in analyzing the convergence behavior of the FlexATC.
\begin{Lemma}\label{LEM-2}
Under Assumptions \ref{ASS1}, \ref{MixingMatrix-1}, and \ref{MixingMatrix-2}, if $\alpha>0$ and $p\in(0,1]$, the sequence $\{\bm{x}^k,\}$ generated by \eqref{Framework-UP-2} satisfies that
\begin{align}\label{lem-1-eq}
&\mathbb{E}\left[\left(\left\|\bm{x}^{k+1}-\bm{x}^\star\right\|^2+\nicefrac{1}{p^2}\left\|\bm{u}^{k+1}-\bm{u}_b^\star\right\|^2\right)|\theta_k\right] \nonumber\\
&\leq \left\|\bm{w}^k-\bm{w}^\star\right\|^2+\left(1-p^2\sigma_m(\bm{B})\right)\nicefrac{1}{p^2}\left\|\bm{u}^k-\bm{u}^\star_b\right\|^2.
\end{align}
\end{Lemma}

{\it Proof}:
Define
$\bm{w}_o^k\triangleq \bm{w}^k-\bm{w}^\star, \bm{x}_o^k\triangleq \bm{x}^k-\bm{x}^\star, \bm{u}_o^k\triangleq \bm{u}^k-\bm{u}_b^\star$,
where $(\bm{x}^{\star},\bm{w}^{\star},\bm{u}_b^\star)$ satisfies the optimal condition \eqref{KKT-Condition} and $\bm{u}_b^\star\in\m{span}(\sqrt{\bm{B}})$. Letting $\bm{z}^k \triangleq \bm{w}^k-\sqrt{\bm{B}}\bm{u}^k, \bm{z}^\star \triangleq \bm{w}^\star-\sqrt{\bm{B}}\bm{u}_b^\star, \bm{z}_o^k\triangleq \bm{z}^k-\bm{z}^\star$, with Assumption \ref{MixingMatrix-2}, it gives that
$\bm{B}\bm{z}^\star = \bm{0} \Longleftrightarrow \sqrt{\bm{B}}\bm{z}^\star = \bm{0}$.
Thus, by \eqref{KKT-Condition} and \eqref{Framework-UP-2}, we can reach the error recursions:
\begin{subequations}\label{Framework-UP-3}
\begin{align}
  \bm{z}_o^k &= \bm{w}_o^k-\sqrt{\bm{B}}\bm{u}_o^k, \label{Framework-UP-3-1}\\
  \bm{x}_o^{k+1} & = \mathrm{prox}_{\alpha R}(\bm{A}_k\bm{z}^k)-\mathrm{prox}_{\alpha R}(\bm{A}\bm{z}^\star), \\
  \bm{u}_o^{k+1} & = \bm{u}_o^k + p \sqrt{\bm{B}_k}\bm{z}^k-\sqrt{\bm{B}}\bm{z}^\star.
\end{align}
\end{subequations}
Note that $\bm{A}_k= \theta_k \bm{A} + (1-\theta_k)\bm{I}$, where $\theta_k\in\{0,1\}$, $\mathop{\m{Prob}}(\theta_k =1) = p$, and $\mathop{\m{Prob}}(\theta_k =0) = 1-p$. Taking the expectation over $\theta_k$, it gives that
\begin{align*}
&\mathbb{E}\left[\left\|\bm{x}_o^{k+1}\right\|^2|\theta_k\right]=p\left\|\mathrm{prox}_{\alpha R}(\bm{A}\bm{z}^k)-\mathrm{prox}_{\alpha R}(\bm{A}\bm{z}^\star)\right\|^2\\
&\quad+(1-p)\left\|\mathrm{prox}_{\alpha R}(\bm{z}^k)-\mathrm{prox}_{\alpha R}(\bm{A}\bm{z}^\star)\right\|^2.
\end{align*}
Since $\bm{B}\bm{z}^\star=\bm{0}$ and $\bm{A}\bm{1}=\bm{1}$, we have $\bm{A}\bm{z}^\star=\bm{z}^\star$. By the nonexpansivity of $\mathrm{prox}_{\alpha R}(\cdot)$, it gives that
\begin{align}\label{proof-thm1-eq1}
\mathbb{E}\left[\left\|\bm{x}_o^{k+1}\right\|^2|\theta_k\right]&\leq p\left\|\bm{A}\bm{z}_o^k\right\|^2+(1-p)\left\|\bm{z}_o^k\right\|^2\nonumber\\
&=\left\|\bm{z}_o^k\right\|^2_{p\bm{A}^2+(1-p)\bm{I}}\leq\left\|\bm{z}_o^k\right\|^2_{\bm{I}-p\bm{B}}.
\end{align}
The last inequality holds as $\bm{A}^2\preceq\bm{I}-\bm{B}$. Next, from \eqref{Framework-UP-3-1}, and expanding the squared norm of $\|\bm{z}_o^k\|_{\bm{I}-p\bm{B}}^2$ and collecting the terms, we can obtain that
\begin{align*}
\left\|\bm{z}_o^k\right\|_{\bm{I}-p\bm{B}}^2&=\left\|\bm{z}_o^k\right\|^2-\left\|\bm{z}_o^k\right\|_{p\bm{B}}^2\\
&=\left\|\bm{w}_o^k-\sqrt{\bm{B}}\bm{u}_o^k\right\|^2-\left\|\bm{z}_o^k\right\|_{p\bm{B}}^2\\
&=\left\|\bm{w}_o^k\right\|^2+\left\|\bm{u}_o^k\right\|_{\bm{B}}^2-2\left\langle\bm{w}_o^k,\sqrt{\bm{B}}\bm{u}_o^k\right\rangle-\left\|\bm{z}_o^k\right\|_{p\bm{B}}^2.
\end{align*}
Therefore, combining with \eqref{proof-thm1-eq1}, it holds that
\begin{align}\label{proof-thm1-eq2}
&\mathbb{E}\left[\left\|\bm{x}_o^{k+1}\right\|^2|\theta_k\right]\nonumber\\
&\leq\left\|\bm{w}_o^k\right\|^2+\left\|\bm{u}_o^k\right\|_{\bm{B}}^2-2\left\langle\bm{w}_o^k,\sqrt{\bm{B}}\bm{u}_o^k\right\rangle-\left\|\bm{z}_o^k\right\|_{p\bm{B}}^2.
\end{align}
Note that $\sqrt{\bm{B}_k}=\theta_k\sqrt{\bm{B}}$, where $\theta_k\in\{0,1\}$, $\mathop{\m{Prob}}(\theta_k =1) = p$. Taking the expectation over $\theta_k$, it gives that
\begin{align}\label{proof-thm1-eq3}
&\mathbb{E}\left[\left\|\bm{u}_o^{k+1}\right\|^2|\theta_k\right]=p\left\|\bm{u}_o^k+p\sqrt{\bm{B}}\bm{z}_o^k\right\|^2+(1-p)\left\|\bm{u}_o^k\right\|^2\nonumber\\
&=\left\|\bm{u}_o^k\right\|^2+p^2\left\|\bm{z}_o^k\right\|^2_{p\bm{B}}+2p^2\left\langle\bm{u}^k_o,\sqrt{\bm{B}}\bm{z}_o^k\right\rangle.
\end{align}
Since $\bm{z}_o^k = \bm{w}_o^k-\sqrt{\bm{B}}\bm{u}_o^k$, we have
\begin{align*}
\langle\bm{u}^k_o,\sqrt{\bm{B}}\bm{z}_o^k\rangle&=\langle\sqrt{\bm{B}}\bm{u}^k_o,\bm{w}_o^k-\sqrt{\bm{B}}\bm{u}_o^k\rangle\\
&=\langle\bm{w}_o^k,\sqrt{\bm{B}}\bm{u}^k_o\rangle-\|\bm{u}_o^k\|^2_{\bm{B}}.
\end{align*}
Thus, combining with \eqref{proof-thm1-eq3}, we have
\begin{align*}
\mathbb{E}[\|\bm{u}_o^{k+1}\|^2|\theta_k]=&\|\bm{u}_o^k\|^2+p^2\|\bm{z}_o^k\|^2_{p\bm{B}}\\
&+2p^2(\langle\bm{w}_o^k,\sqrt{\bm{B}}\bm{u}^k_o\rangle-\|\bm{u}_o^k\|^2_{\bm{B}}).
\end{align*}
Finally, combining with \eqref{proof-thm1-eq2}, it holds that
\begin{align*}
&\mathbb{E}\left[\left(\left\|\bm{x}_o^{k+1}\right\|^2+\nicefrac{1}{p^2}\left\|\bm{u}_o^{k+1}\right\|^2\right)|\theta_k\right]\\
&\leq \left\|\bm{w}_o^k\right\|^2+\left\|\bm{u}_o^k\right\|_{\bm{B}}^2-2\left\langle\bm{w}_o^k,\sqrt{\bm{B}}\bm{u}_o^k\right\rangle-\left\|\bm{z}_o^k\right\|_{p\bm{B}}^2\\
&\quad+\nicefrac{1}{p^2}\left\|\bm{u}_o^k\right\|^2+\left\|\bm{z}_o^k\right\|^2_{p\bm{B}}+2\left\langle\bm{w}_o^k,\sqrt{\bm{B}}\bm{u}^k_o\right\rangle-2\left\|\bm{u}_o^k\right\|^2_{\bm{B}}\\
&= \left\|\bm{w}_o^k\right\|^2+\nicefrac{1}{p^2}\left\|\bm{u}_o^k\right\|^2-\left\|\bm{u}_o^k\right\|^2_{\bm{B}}.
\end{align*}
Next, we consider the lower bound of $\left\|\bm{u}_o^k\right\|^2_{\bm{B}}$. From \eqref{Framework-UP-2-3}, it gives that, for $\bm{u}^0=\bm{0}$, $\bm{u}^k\in\m{span}(\sqrt{\bm{B}})$ for any $k\geq0$. Considering that $\bm{u}^\star_b\in\m{span}(\sqrt{\bm{B}})$. It holds that $\bm{u}_o^k\in\m{span}(\sqrt{\bm{B}})$ for any $k\geq0$. Therefore, it gives that
$\left\|\bm{u}_o^k\right\|^2_{\bm{B}}\geq\sigma_m(\bm{B})\left\|\bm{u}_o^k\right\|^2$, which implies that Lemma \ref{LEM-2} holds.
\hfill{$\square$}

\subsection{Sublinear convergence rate}
Using Lemmas \ref{LEM-1} and \ref{LEM-2}, we give the following theorem.
\begin{Theorem}\label{THM-1}
Under Assumptions \ref{ASS1}, \ref{MixingMatrix-1}, and \ref{MixingMatrix-2}, if $\alpha\in(0,\nicefrac{2}{L})$ and $p\in(0,1]$, then $\lim_{k\rightarrow\infty}\|\nabla F(\bm{x}^k)-\nabla F(\bm{x}^\star)\|^2=0$ and $\lim_{k\rightarrow\infty}\|\bm{u}^k-\bm{u}_b^\star\|^2=0$, almost surely and in quadratic mean. In addition, letting $\bar{\bm{X}}^{K}=\frac{1}{K}\sum_{k=0}^{K-1}\bm{x}^k$ and $\bar{\bm{U}}^{K}=\frac{1}{K}\sum_{k=0}^{K-1}\bm{u}^k$, it holds that
\begin{align*}
\mathbb{E}\left[\left\|\nabla F(\bar{\bm{X}}^{K})-\nabla F(\bm{x}^\star)\right\|^2
+\left\|\bar{\bm{U}}^K-\bm{u}^\star_b\right\|^2
\right]=\mathcal{O}\left(\frac{1}{K}\right).
\end{align*}
\end{Theorem}

{\it Proof}:
From the definition of $\bm{w}^k_o$, we have
\begin{align*}
\left\|\bm{w}^k_o\right\|^2&=\left\|\bm{x}_o^k\right\|^2+\alpha^2\|\nabla F(\bm{x}^k)-\nabla F(\bm{x}^\star)\|^2\\
&\quad-2\alpha\left\langle\bm{x}_o^k,\nabla F(\bm{x}^k)-\nabla F(\bm{x}^\star)\right\rangle\nonumber\\
&\leq\left\|\bm{x}_o^k\right\|^2-\alpha(\nicefrac{2}{L}-\alpha)\|\nabla F(\bm{x}^k)-\nabla F(\bm{x}^\star)\|^2,
\end{align*}
where the second inequality follows from the convexity and smoothness of $f_i$, which is equivalent to
$\langle\bm{x}_o^k,\nabla F(\bm{x}^k)-\nabla F(\bm{x}^\star)\rangle\geq\nicefrac{1}{L}\|\nabla F(\bm{x}^k)-\nabla F(\bm{x}^\star)\|^2$.
Let
\begin{align*}
\Phi^k&\triangleq\left\|\bm{x}^{k}-\bm{x}^\star\right\|^2+\nicefrac{1}{p^2}\left\|\bm{u}^{k}-\bm{u}_b^\star\right\|^2,\\
\Psi^k& \triangleq \left\|\nabla F(\bm{x}^k)-\nabla F(\bm{x}^\star)\right\|^2 + \left\|\bm{u}^k-\bm{u}^\star_b\right\|^2.
\end{align*}
From \eqref{lem-1-eq}, one has
\begin{align}\label{newproof-eq1}
\mathbb{E}[\Phi^{k+1}|\theta_k]\leq &\Phi^k-\alpha(\nicefrac{2}{L}-\alpha)\|\nabla F(\bm{x}^k)-\nabla F(\bm{x}^\star)\|^2\nonumber\\
& - \sigma_m(\bm{B})\|\bm{u}^k-\bm{u}^\star_b\|^2\nonumber\\
\leq&\Phi^k -\varrho \Psi^k,
\end{align}
where $\varrho=\min\{\alpha(\nicefrac{2}{L}-\alpha),\sigma_m(\bm{B})\}$.
Using classical results on almost supermartingale convergence theorem \cite[Proposition A.4.5, pp. 464]{BOOK2015} \cite{Robbins1971}, it follows from \eqref{newproof-eq1} that $\Phi^{k}$ converges almost surely to a  random variable $\Phi^\infty$ and that
$\sum_{k=0}^{\infty}\Psi^k<\infty$ almost surely. Thus,
$\lim_{k\rightarrow\infty}\Psi^k=0$ almost surely, i.e.,
\begin{align*}
&\lim_{k\rightarrow\infty}\|\nabla F(\bm{x}^k)-\nabla F(\bm{x}^\star)\|^2=0, \text{ almost surely},\\
&\lim_{k\rightarrow\infty}\|\bm{u}^k-\bm{u}_b^\star\|^2=0, \text{ almost surely}.
\end{align*}
Moreover, taking full expectation in \eqref{newproof-eq1}, it gives that
\begin{align}\label{PROOF-THM1-EQ1}
&\varrho\mathbb{E}\left[\Psi^{k}\right]\leq \mathbb{E}\left[\Phi^{k}\right] - \mathbb{E}\left[\Phi^{k+1}\right].
\end{align}
Summing the inequality \eqref{PROOF-THM1-EQ1} over $k=0,1,\cdots,\infty$, we have
$\sum_{k=0}^{\infty}\mathbb{E}\left[\Psi^{k}\right]\leq \nicefrac{1}{\varrho}\Phi^0<\infty$.
Therefore,
$\lim_{k\rightarrow\infty}\mathbb{E}\left[\Psi^{k}\right]=0$,
which implies that $\lim_{k\rightarrow\infty}\Psi^{k}=0$ in quadratic mean, i.e.,
\begin{align*}
\lim_{k\rightarrow\infty}\|\nabla F(\bm{x}^k)-\nabla F(\bm{x}^\star)\|^2=0,\
\lim_{k\rightarrow\infty}\|\bm{u}^k-\bm{u}_b^\star\|^2=0
\end{align*}
in quadratic mean. Since $\|\cdot\|^2$ is convex, summing \eqref{PROOF-THM1-EQ1} over $k=0,1,\cdots,K-1$, it gives that
\begin{align*}
\mathbb{E}\left[\left\|\nabla F(\bar{\bm{X}}^{K})-\nabla F(\bm{x}^\star)\right\|^2 +
                 \left\|\bar{\bm{U}}^K-\bm{u}^\star_b\right\|^2
\right]\leq \frac{\Phi^0}{\varrho K},
\end{align*}
where $\bar{\bm{X}}^{K}=\frac{1}{K}\sum_{k=0}^{K-1}\bm{x}^k$ and $\bar{\bm{U}}^{K}=\frac{1}{K}\sum_{k=0}^{K-1}\bm{u}^k$.
\hfill{$\square$}

With $\mu=0$, Theorem \ref{THM-1} shows that FlexATC is sublinear convergence. However, achieving communication acceleration requires assuming $\mu>0$. In the next two subsections, we will give a more detailed analysis.

\subsection{Linear convergence rate}

In the strongly convex setting, applying Lemmas \ref{LEM-1} and \ref{LEM-2}, we present the following theorem to demonstrate that FlexATC can achieve linear convergence.
\begin{Theorem}\label{THM-2}
Under Assumptions \ref{ASS1}, \ref{MixingMatrix-1}, and \ref{MixingMatrix-2}, if $\mu>0$, $\alpha\in(0,\nicefrac{2}{L})$ and $p\in(0,1]$, then it holds that
\begin{align*}
\mathbb{E}\left[\left\|\bm{x}^{k}-\bm{x}^\star\right\|^2\right]\leq\left(\left\|\bm{x}^{0}-\bm{x}^\star\right\|^2+\nicefrac{1}{p^2}\left\|\bm{u}_b^\star\right\|^2\right) \zeta^{k},
\end{align*}
with
$\zeta=\max\{(1-\alpha L)^2,(1-\alpha \mu)^2,1-p^2\sigma_m(\bm{B})\}<1$. Moreover, $\{(\bm{x}^k,\bm{u}^k)\}_{k\geq0}$ converges to $(\bm{x}^\star,\bm{u}_b^\star)$ almost surely.
\end{Theorem}

{\it Proof}:
Let $D_F^k\triangleq\|\nabla F(\bm{x}^k)-\nabla F(\bm{x}^\star)\|^2$.
From the cocoercivity of $(\nabla F-\mu \bm{I})$ and \cite{BOOK2015}, we have
$$
\langle\nabla F(\bm{x}^k)-\nabla F(\bm{x}^\star),\bm{x}^k-\bm{x}^\star\rangle
\geq\frac{L \mu}{L+\mu}\|\bm{x}^k-\bm{x}^\star\|^2+\frac{1}{L+\mu}D_F^k.
$$
Thus, it gives that
\begin{align*}
\left\|\bm{w}^k_o\right\|^2
&\leq\left(1-\frac{2\alpha L \mu}{L+\mu}\right)\left\|\bm{x}^k-\bm{x}^\star\right\|^2+\left(\alpha^2-\frac{2\alpha}{L+\mu}\right)D_F^k.
\end{align*}
If $\alpha\leq\frac{2}{L+\mu}$, since
$D_F^k\geq\mu\|\bm{x}^k-\bm{x}^\star\|$, it gives that
\begin{align*}
\left\|\bm{w}^k_o\right\|^2
&\leq\left(1-\frac{2\alpha L \mu}{L+\mu}+\left(\alpha^2-\frac{2\alpha}{L+\mu}\right)\mu^2\right)\left\|\bm{x}^k-\bm{x}^\star\right\|^2\\
&=\left(1-\alpha\mu\right)^2\left\|\bm{x}^k-\bm{x}^\star\right\|^2.
\end{align*}
If $\alpha\geq\frac{2}{L+\mu}$, since
$D_F^k\leq L\|\bm{x}^k-\bm{x}^\star\|$, it gives that
\begin{align*}
\left\|\bm{w}^k_o\right\|^2
&\leq\left(1-\frac{2\alpha L \mu}{L+\mu}+\left(\alpha^2-\frac{2\alpha}{L+\mu}\right)L^2\right)\left\|\bm{x}^k-\bm{x}^\star\right\|^2\\
&=\left(1-\alpha L\right)^2\left\|\bm{x}^k-\bm{x}^\star\right\|^2.
\end{align*}
Therefore, we have $\|\bm{w}^k_o\|^2\leq\max\{(1-\alpha L)^2,(1-\alpha \mu)^2\}\|\bm{x}^k_o\|^2$.
Then, combining it with \eqref{lem-1-eq}, it holds that
\begin{align*}
&\mathbb{E}\left[\left(\left\|\bm{x}^{k+1}-\bm{x}^\star\right\|^2+\nicefrac{1}{p^2}\left\|\bm{u}^{k+1}-\bm{u}_b^\star\right\|^2\right)|\theta_k\right]\\
&\leq \max\{\left(1-\alpha L\right)^2,\left(1-\alpha \mu\right)^2\}\left\|\bm{x}^k-\bm{x}^\star\right\|^2\\
&\quad+\left(1-p^2\sigma_m(\bm{B})\right)\left(\nicefrac{1}{p^2}\left\|\bm{u}_o^k-\bm{u}^\star_b\right\|^2\right).
\end{align*}
Thus, it deduces that
$\mathbb{E}[\Phi^{k+1}|\theta_k]\leq \zeta\Phi^k$,
where $\zeta=\max\{(1-\alpha L)^2,(1-\alpha \mu)^2,1-p^2\sigma_m(\bm{B})\}$. Using classical results on almost supermartingale convergence theorem \cite[Proposition A.4.5, pp. 464]{BOOK2015} \cite{Robbins1971}, it holds that $\lim_{k\rightarrow\infty}\Phi^{k}=0$ almost surely. Almost sure convergence of $\{(\bm{x}^k,\bm{u}^k)\}_{k\geq0}$ follows. Taking full expectation and unrolling the recurrence, one has that
\begin{align*}
\mathbb{E}\left[\Phi^{k+1}\right]\leq \zeta^{k+1}\Phi^0\Rightarrow \mathbb{E}\left[\left\|\bm{x}^{k+1}-\bm{x}^\star\right\|^2\right]\leq\zeta^{k+1}\Phi^0.
\end{align*}
Thus, the proof is completed.
\hfill{$\square$}

Under the strongly convex setting, Theorem \ref{THM-2} offers a separated linear convergence rate of FlexATC, whose dependency on the agents' loss functions and the network topology is decoupled, and the term $\zeta_c=\max\{(1-\alpha L)^2,(1-\alpha \mu)^2\}$
matches the convergence rate exhibited by the proximal gradient algorithm \cite{ProxGD1,ProxGD2} when applied to the problem \eqref{EQ:Problem1}.

Next, we further analyze the effect of local updates on the convergence rate. We discuss the following two cases.

$\bullet$ \textbf{Case 1 -- when $\nicefrac{(1-\zeta_c)}{\sigma_m(\bm{B})}<1$}: In real-world scenarios, this condition holds true when $\kappa=\nicefrac{L}{\mu}$ is large (i.e., for ill-condition problems, which are precisely the ones that need acceleration). In this case, it can be verified that the convergence rate of FlexATC is
$$
\zeta\equiv\max\{(1-\alpha L)^2,(1-\alpha \mu)^2\}, \text{ when } p^2\in\left[\frac{1-\zeta_c}{\sigma_m(\bm{B})},1\right].
$$
Since the stepsize $\alpha=\mathcal{O}(\nicefrac{1}{L})$ is independent of the number of local updates $\nicefrac{1}{p}$, the convergence rate of FlexATC stays unaffected when $p$ is reduced from $1$ to $\sqrt{\nicefrac{(1-\zeta_c)}{\sigma_m(\bm{B})}}$ (i.e., the introduction of a local update mechanism will not cause any loss). This explains why we can often skip communication rounds without any negative impact on the convergence rate, essentially allowing us to ``skip" communication at no cost.

$\bullet$ \textbf{Case 2 -- when $\nicefrac{(1-\zeta_c)}{\sigma_m(\bm{B})} \geq 1$}: In real-world scenarios this condition holds true when $\kappa=\nicefrac{L}{\mu}$ is not too large. In this case, the convergence rate of FlexATC is
$\zeta = 1 - p^2 \sigma_m(\bm{B})$,
i.e., the optimal efficient approach is to perform communication at each iteration, corresponding to setting $p = 1$. In fact, for this well-conditioned problem ($\kappa$ is not large), additional acceleration is not needed because the improvement from $\kappa$ to $\sqrt{\kappa}$ is not significant.

Note that, due to the flexibility of FlexATC, we can make $\nicefrac{(1-\zeta_c)}{\sigma_m(\bm{B})}<1$ hold for any $\kappa$ by selecting $\bm{B}$. For example, similar as MG-ED \cite{Yuan2021} and MG-SONATA \cite{SONATA}, we can use muti-gossip communications to make $\mathcal{O}(\sigma_m(\bm{B}))=1$. In other words, combining local updates with multi-gossip can accelerate any situation.

\begin{table*}[!t]
\renewcommand\arraystretch{2}
\begin{center}
\caption{Comparison with existing strongly convex convergence rates of SOTA ATC-based distributed algorithms.\\ Iter. = Iteration, Comm. = Communication, Dec. = Deceleration, Acc. = Acceleration. }
\scalebox{0.66}{
\begin{tabular}{ccccccccccc}
\hline
\multirow{2}{*}{Algorithm} & \multirow{2}{*}{$\bm{B}$} &\multirow{2}{*}{ $\mathcal{O}\left(\sigma_m(\bm{B})\right)$ }&\multirow{2}{*}{ $N$ }&  \# of Iter.: $p=1$&\# of Iter.: $p\in\left[\sqrt{\frac{(1-\zeta_c)}{\sigma_m(\bm{B})}},1\right]$& No Iter. Dec.&  \# of Comm.: $p=1$&\# of Comm.: $p\in\left[\sqrt{\frac{(1-\zeta_c)}{\sigma_m(\bm{B})}},1\right]$& Acc. Comm.  \\
&&&&literature&\textbf{this paper}&\textbf{this paper}&literature&\textbf{this paper}&\textbf{this paper}\\
\hline
NIDS/ED/D2&$c(\bm{I}-\bm{W})$&$1-\rho$& $1$ & $\mathcal{O}\left(\frac{\kappa}{1-\rho}\log\nicefrac{1}{\epsilon}\right)$&$\mathcal{O}\left(\frac{\kappa}{1-\rho}\log\nicefrac{1}{\epsilon}\right)$& \textcolor[rgb]{0,0.6,0}{\Checkmark}& $\mathcal{O}\left(\frac{\kappa}{1-\rho}\log\nicefrac{1}{\epsilon}\right)$&$\mathcal{O}\left(\sqrt{\frac{\kappa}{1-\rho}}\log\nicefrac{1}{\epsilon}\right)$, $p=\frac{1}{\sqrt{\kappa(1-\rho)}}$&\textcolor[rgb]{0,0.6,0}{\Checkmark}\\
MG-ED&$\nicefrac{1}{2}(\bm{I}-\bm{W}^N)$&$1$ & $\mathcal{O}\left(\frac{1}{\sqrt{1-\rho}}\right)$&
$\mathcal{O}\left(\kappa\log\nicefrac{1}{\epsilon}\right)$&$\mathcal{O}\left(\kappa\log\nicefrac{1}{\epsilon}\right)$&\textcolor[rgb]{0,0.6,0}{\Checkmark}
&$\mathcal{O}\left(\frac{\kappa}{\sqrt{1-\rho}}\log\nicefrac{1}{\epsilon}\right)$&$\mathcal{O}\left(\sqrt{\frac{\kappa}{1-\rho}}\log\nicefrac{1}{\epsilon}\right)$, $p=\frac{1}{\sqrt{\kappa}}$&\textcolor[rgb]{0,0.6,0}{\Checkmark}\\
ATC-GT&$(\bm{I}-\bm{W})^2$  &$(1-\rho)^2$&$\mathcal{O}(1)$& $\mathcal{O}\left(\frac{\kappa}{(1-\rho)^2}\log\nicefrac{1}{\epsilon}\right)$&$\mathcal{O}\left(\frac{\kappa}{(1-\rho)^2}\log\nicefrac{1}{\epsilon}\right)$&\textcolor[rgb]{0,0.6,0}{\Checkmark}&
$\mathcal{O}\left(\frac{\kappa}{(1-\rho)^2}\log\nicefrac{1}{\epsilon}\right)$&$\mathcal{O}\left(\frac{\sqrt{\kappa}}{1-\rho}\log\nicefrac{1}{\epsilon}\right)$, $p=\frac{1}{\sqrt{\kappa(1-\rho)^2}}$&\textcolor[rgb]{0,0.6,0}{\Checkmark}\\
MG-SONATA& $(\bm{I}-\bm{W}^N)^2$  &$1$ & $\mathcal{O}\left(\frac{1}{\sqrt{1-\rho}}\right)$ &$\mathcal{O}\left(\kappa\log\nicefrac{1}{\epsilon}\right)$&$\mathcal{O}\left(\kappa\log\nicefrac{1}{\epsilon}\right)$& \textcolor[rgb]{0,0.6,0}{\Checkmark} &$\mathcal{O}\left(\frac{\kappa}{\sqrt{1-\rho}}\log\nicefrac{1}{\epsilon}\right)$&$\mathcal{O}\left(\sqrt{\frac{\kappa}{1-\rho}}\log\nicefrac{1}{\epsilon}\right)$, $p=\frac{1}{\sqrt{\kappa}}$&\textcolor[rgb]{0,0.6,0}{\Checkmark}\\
\hline
\end{tabular}}
\label{Table-complexity-2}
\end{center}
\end{table*}

\subsection{Communication complexity analysis}
In this subsection, we consider the communication complexity of FlexATC \eqref{Framework-UP} under strongly convexity. Note that
$\zeta_c=\max\{(1-\alpha L)^2,(1-\alpha \mu)^2\}\leq 1-\alpha \mu$.
It holds that
$\zeta\leq\max\{1-\alpha \mu,1-p^2\sigma_m(\bm{B})\}$.
Letting $\kappa=\nicefrac{L}{\mu}$ and setting $\alpha=\nicefrac{1}{L}$, we have
$\zeta\leq\max\{1-\nicefrac{1}{\kappa},1-p^2\sigma_m(\bm{B})\}$.
Thus, for the desired accuracy level $\epsilon>0$, when $\mathbb{E}[\|\bm{x}^{k+1}-\bm{x}^\star\|^2]\leq\epsilon \Phi^0$, the expected iteration complexity of FlexATC with $\mu>0$ is
$$
\text{\# of Iter.: } \mathcal{O}\left(\max\left\{\kappa,\frac{1}{p^2\sigma_m(\bm{B})}\right\}\log\nicefrac{1}{\epsilon}\right).
$$
When $\nicefrac{1}{\sqrt{\kappa\sigma_m(\bm{B})}}\leq1$, setting $p\in[\nicefrac{1}{\sqrt{\kappa\sigma_m(\bm{B})}},1]$, it holds that the iteration complexity of FlexATC is
$$
\text{\# of Iter.: } \mathcal{O}\left(\kappa\log\nicefrac{1}{\epsilon}\right), \text{\ {unchanged}}.
$$
Thus, the introduction of a local update mechanism will not affect the convergence rate. Given that communication occurs with probability $p$ in each iteration, the expected communication complexity is
\begin{align}\label{Comm} \text{\# of Comm.: } & \mathcal{O}\left(N\left(p\kappa+\frac{1}{p\sigma_m(\bm{B})}\right)\log\nicefrac{1}{\epsilon}\right),
\end{align}
where $N$ is the order of the polynomial functions $\bm{B}$ with respect to $\bm{W}$. If $\nicefrac{1}{\sqrt{\kappa\sigma_m(\bm{B})}}\leq1$, i.e., $\kappa\geq\nicefrac{1}{\sigma_m(\bm{B})}$, letting $p=1$ and $p=\nicefrac{1}{\sqrt{\kappa\sigma_m(\bm{B})}}$, respectively, from \eqref{Comm}, we have
\begin{align*}\left\{
\begin{array}{rl}
    p=1: &  \mathcal{O}\left(N\frac{\kappa}{\sigma_m(\bm{B})}\log\nicefrac{1}{\epsilon}\right),\\
    p=\frac{1}{\sqrt{\kappa\sigma_m(\bm{B})}}: & \mathcal{O}\left(N\sqrt{\frac{\kappa}{\sigma_m(\bm{B})}}\log\nicefrac{1}{\epsilon}\right), \text{\ {improved}}.
       \end{array}\right.
\end{align*}
Therefore, we conclude that local updates can accelerate the communication of ATC-based distributed algorithms without reducing convergence rates.

To elucidate this phenomenon further, we specialize FlexATC \eqref{Framework-UP} to NIDS/ED/D2 \cite{NIDS,ExactDiffusion,D2}, MG-ED \cite{Yuan2021}, ATC-GT \cite{AugDGM1,AugDGM2}, and MG-SONATA \cite{SONATA}, as presented in Table \ref{Table-complexity-2}. Our analysis reveals that not only local gradient steps demonstrably enhance the communication efficiency for NIDS/ED/D2, MG-ED, ATC-GT, and MG-SONATA, without increasing the number of iteration rounds, but they also enable NIDS/ED/D2, MG-ED, and MG-SONATA to achieve the optimal communication complexity, as established in \cite{Scaman2017}.

\begin{figure}[!t]
\centering
\subfigure{
\includegraphics[width=0.48\linewidth]{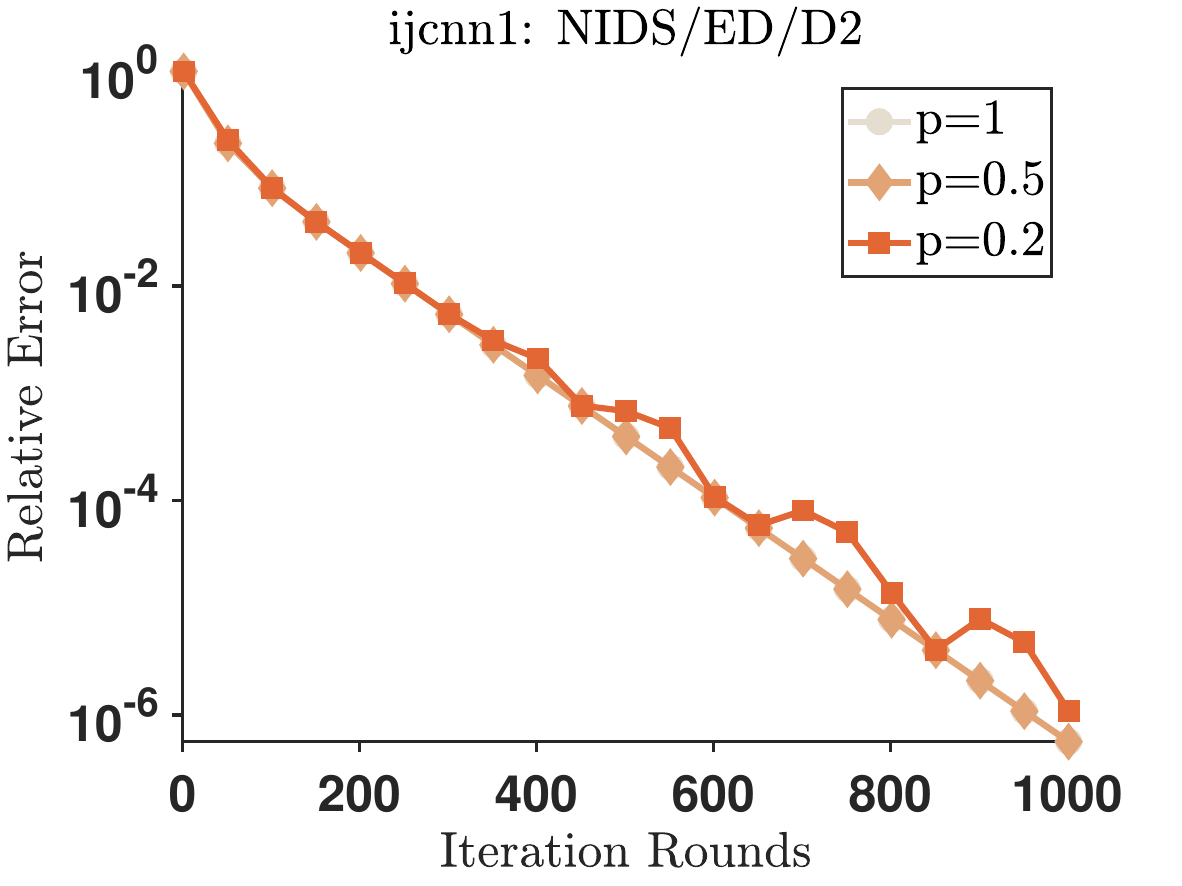}}\hspace{-8pt}
\subfigure{
\includegraphics[width=0.48\linewidth]{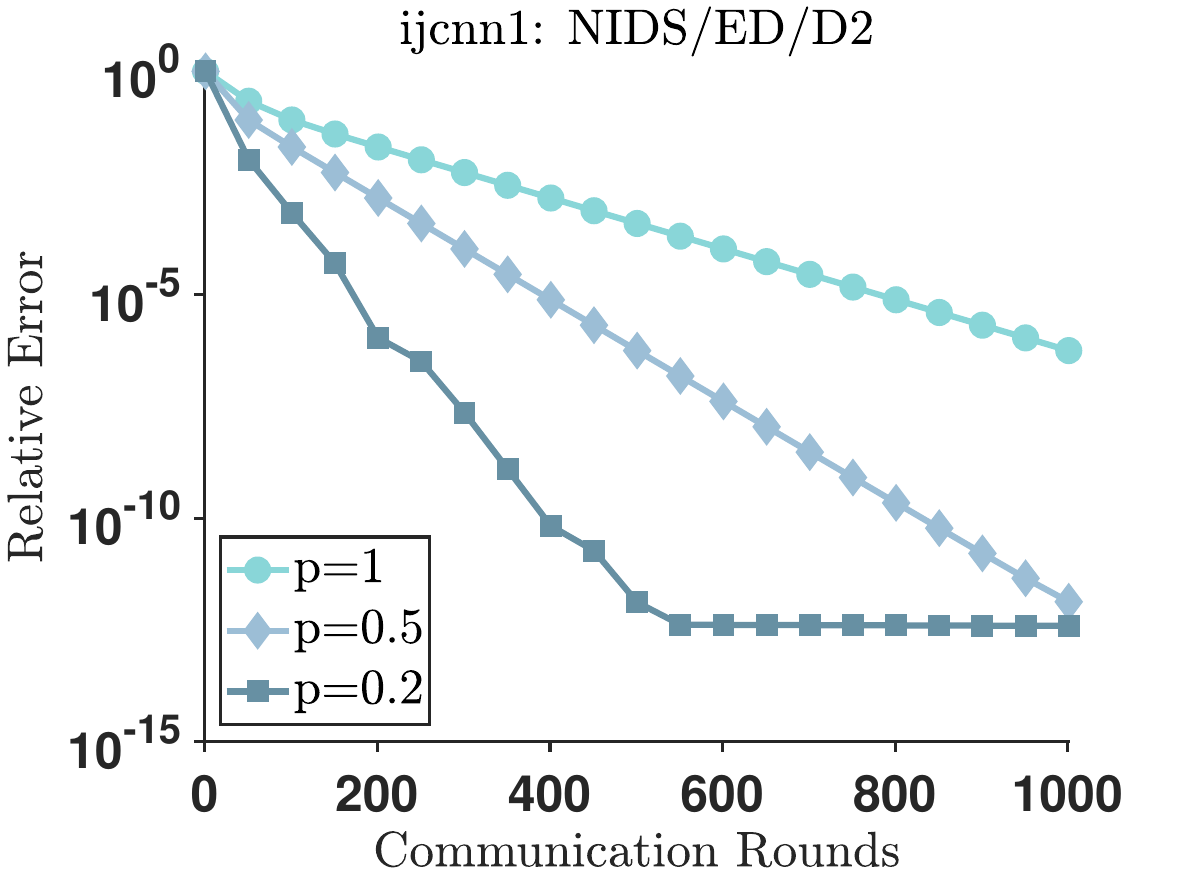}}\hspace{-8pt}
\subfigure{
\includegraphics[width=0.48\linewidth]{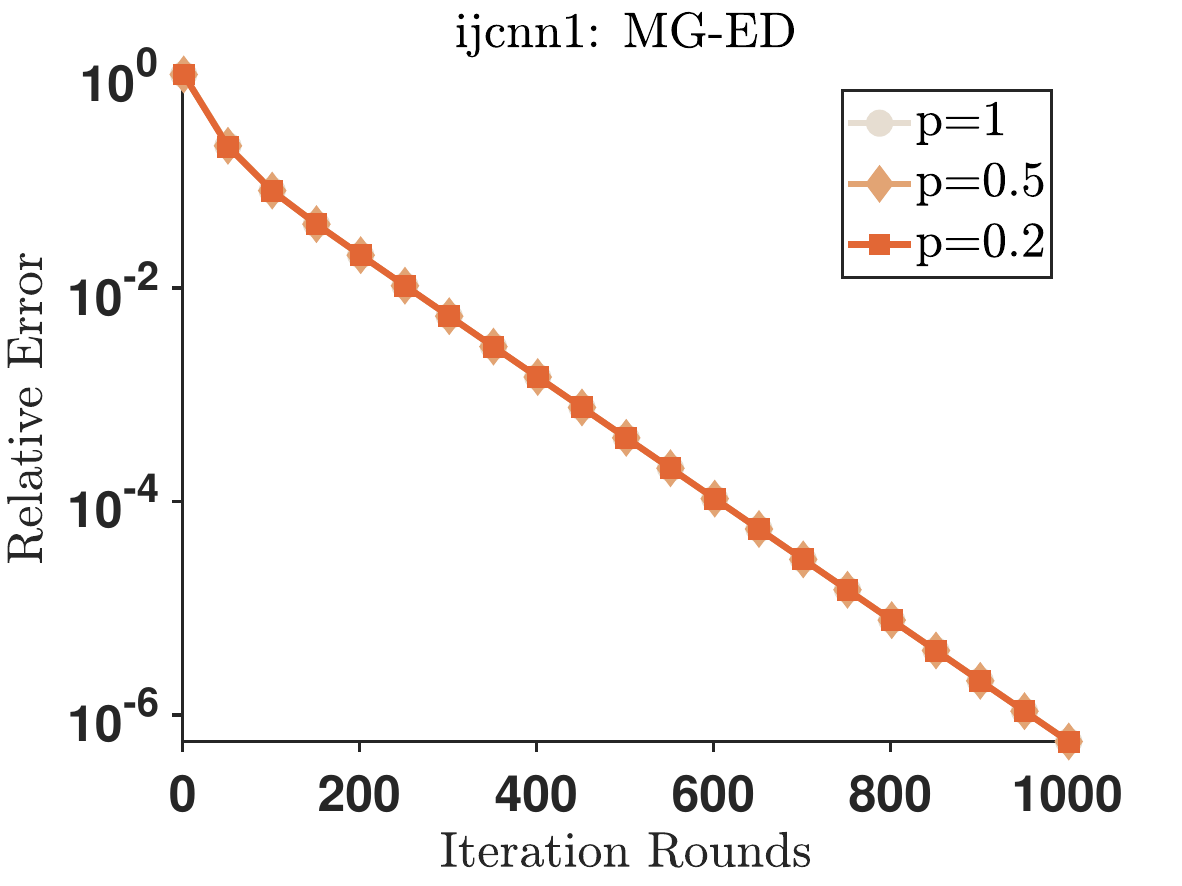}}\hspace{-8pt}
\subfigure{
\includegraphics[width=0.48\linewidth]{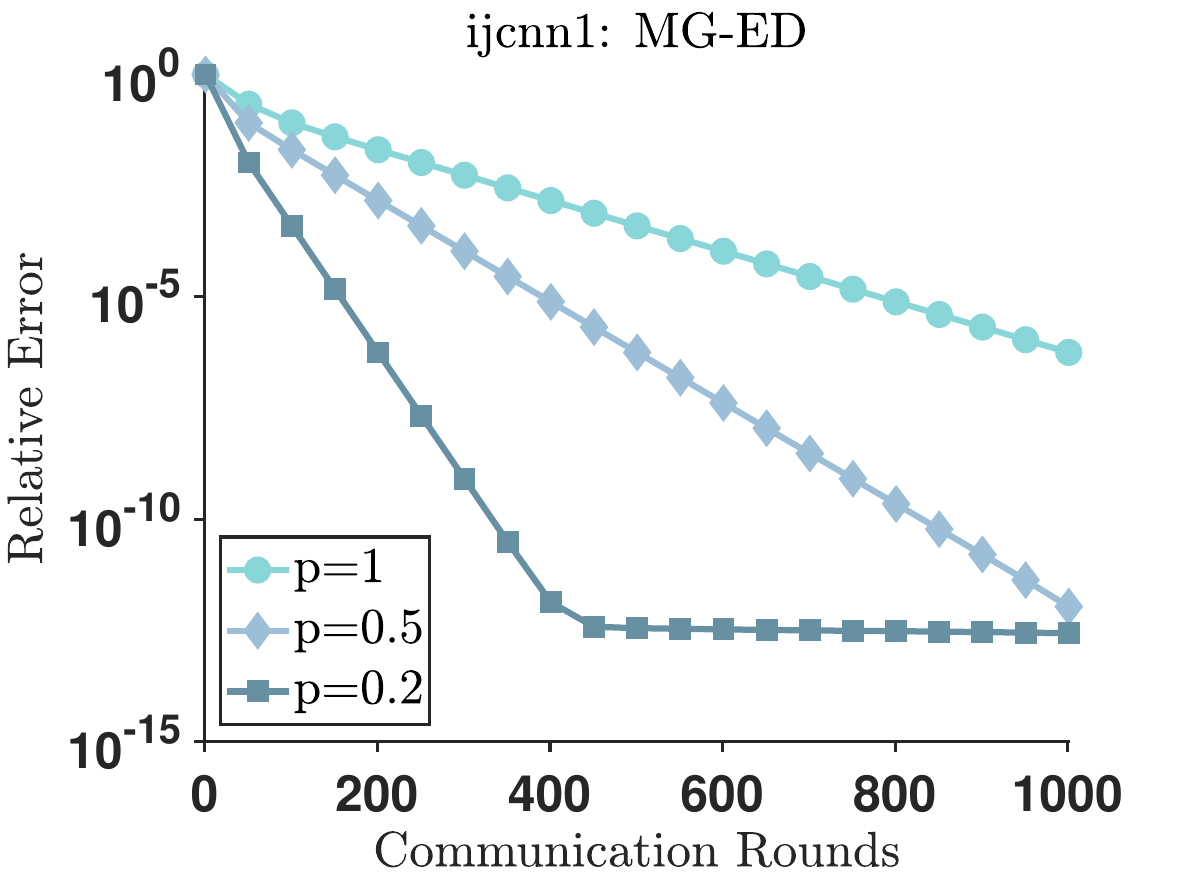}}
\subfigure{
\includegraphics[width=0.48\linewidth]{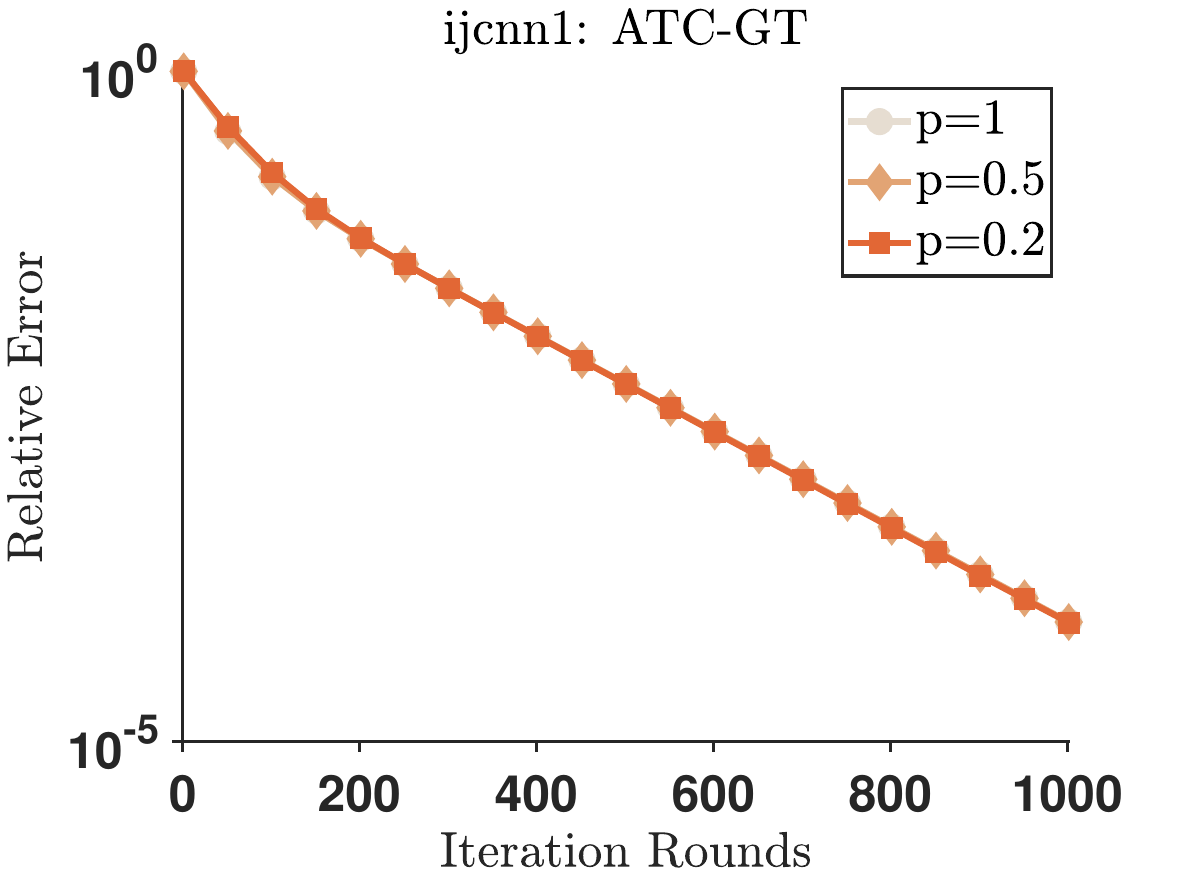}}\hspace{-8pt}
\subfigure{
\includegraphics[width=0.48\linewidth]{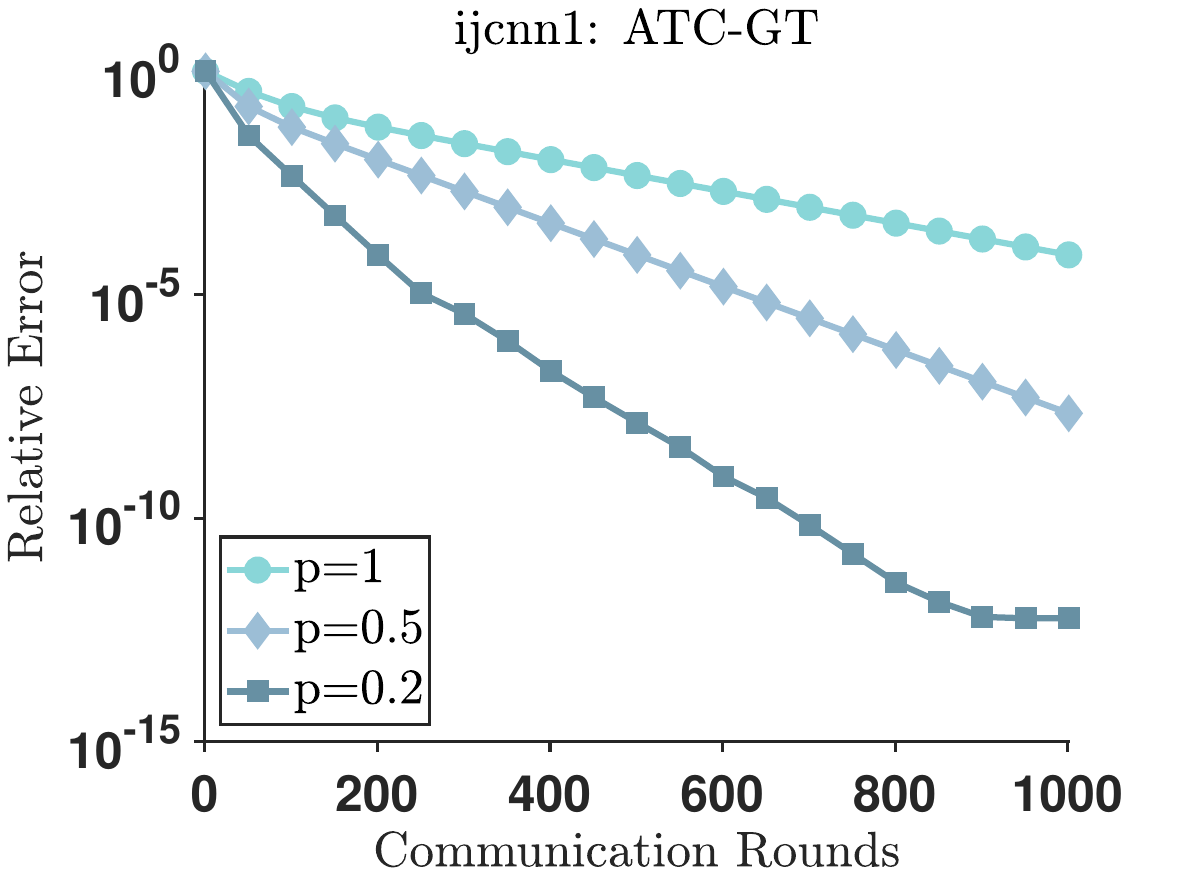}}\hspace{-8pt}
\subfigure{
\includegraphics[width=0.48\linewidth]{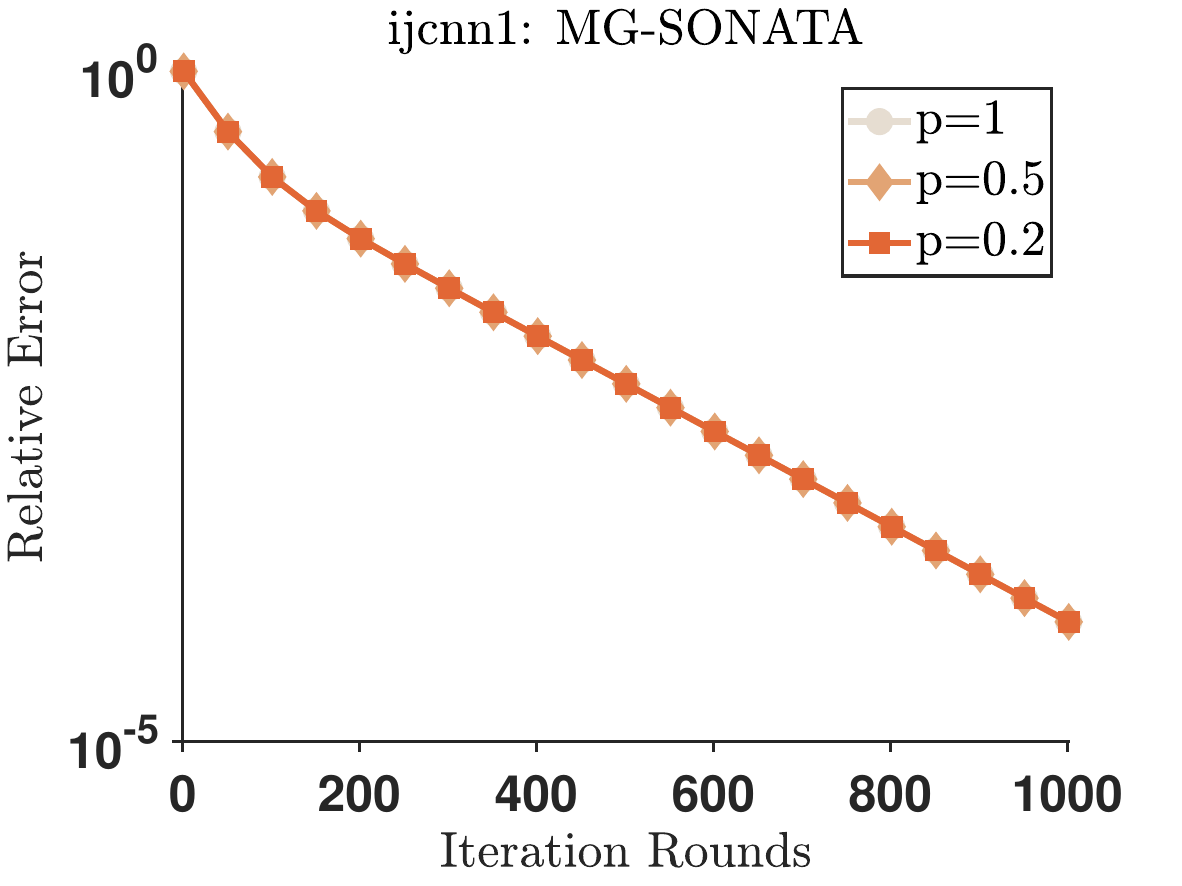}}\hspace{-8pt}
\subfigure{
\includegraphics[width=0.48\linewidth]{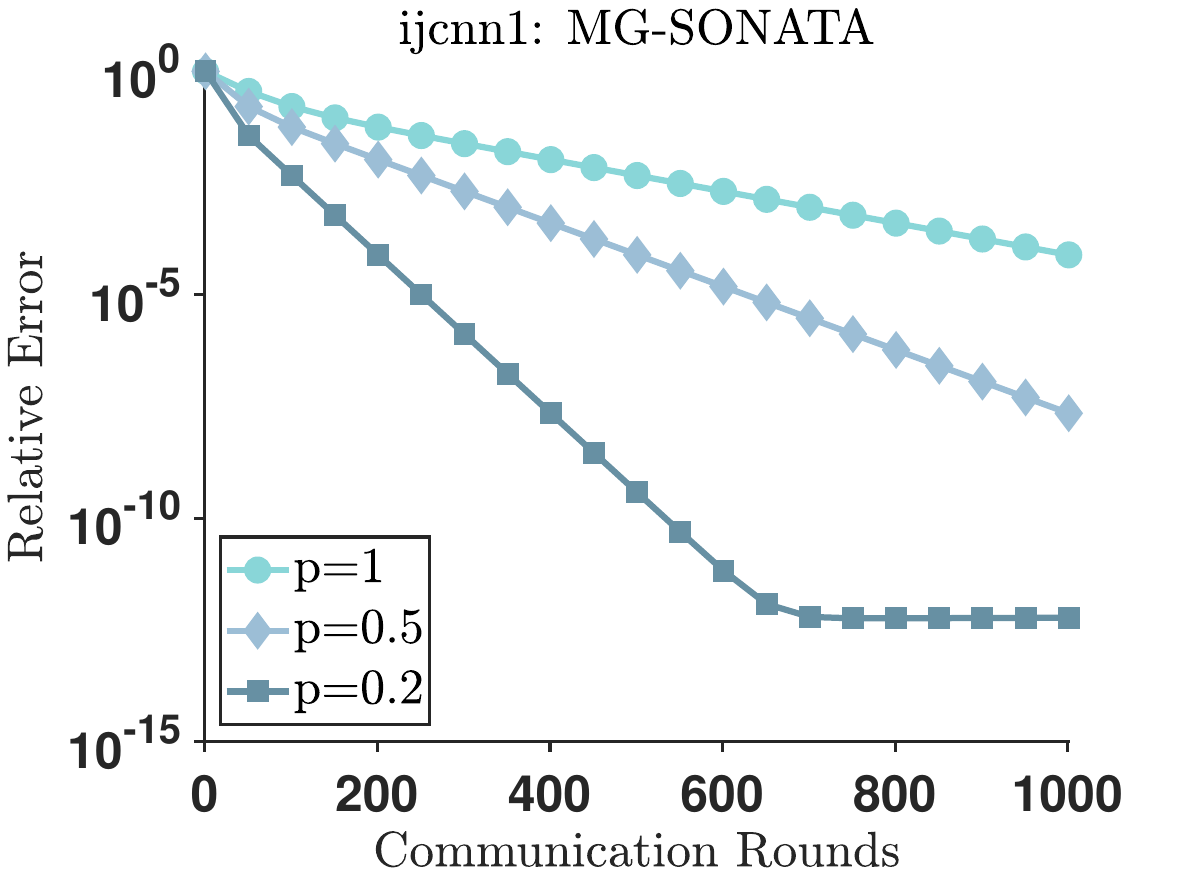}}
\caption{The numerical results over ijcnn1 dataset for communication acceleration are shown, with the relative error $\|\bm{x}^k-\bm{x}^\star\|/\|\bm{x}^\star\|$ plotted against the number of iteration and communication rounds.}
\label{Sim:case1}
\end{figure}

\section{Numerical experiment results}\label{SEC5}
In this numerical experiment, we validate the proposed framework and theoretical results using the real-world datasets ijcnn1 from the widely used LIBSVM \cite{LibSVM}. Choose the local loss function
$f_i(x)=\nicefrac{1}{m_i}\sum_{j=1}^{m_{i}} \ln (1+e^{-(\mathcal{X}_{i j}\tr {x})\mathcal{Y}_{i j}})+\nicefrac{\gamma_1}{2}\|x\|^2$
and the global loss function $r(x) =\gamma_2 \|x\|_1$. Set $\gamma_1=0.01$ and $\gamma_2=0.01$.
In this setup, each node holds its own local training data \((\mathcal{X}_{ij}, \mathcal{Y}_{ij}) \in \mathbb{R}^d \times \{-1, 1\}\), where \(j = 1, \dots, m_i\), consisting of feature vectors \(\mathcal{X}_{ij}\) and their corresponding binary class labels \(\mathcal{Y}_{ij}\). The data \((\mathcal{X}_{ij}, \mathcal{Y}_{ij}) \in \mathbb{R}^d \times \{-1, 1\}\) used in this experiment are sourced from ijcnn1. For the ijcnn1 dataset, the data dimensions are \((d, \sum_{i=1}^{n} m_i) = (22, 49950)\). The training samples are randomly and uniformly distributed across all \(n\) agents. The communication topology is set up as a random connected network with $50$ nodes, and the probability of a communication link between any two nodes is $0.1$.

We compare the iteration and communication performance of ATC-based algorithms, such as ED/NIDS/D2, MG-ED, ATC-GT, and MG-SONATA, under different values of $p$. The specific choices for $\bm{A}$ and $\bm{B}$ of these algorithms are shown in Table \ref{Table-relation}. Set the stepsize $\alpha=\nicefrac{1}{L}$. Fig. \ref{Sim:case1} shows that FlexATC with $p=1,0.5,0.2$ have the same iteration performance. It implies that skipping some communication will not affect the convergence rate. In addition, as shown in Fig. \ref{Sim:case1}, these ATC-based algorithms achieve communication acceleration through the parameter $p$. Through this experiment, we show that local updates can accelerate the communication of ATC-based algorithms.

\section{Conclusion}\label{SEC6}
We introduced a novel algorithmic framework called FlexATC featuring local updates for composite distributed optimization, which unifies numerous ATC-based algorithms. Leveraging this ATC framework, we conducted a unified convergence analysis and established both sublinear and linear convergence rates under general convexity and strong convexity, respectively. Importantly, the convergence stepsize condition of FlexATC remains independent of both the network topology and the number of local updates. Moreover, we demonstrated that the established linear convergence rate is decoupled, with the function component aligning with the centralized proximal gradient descent algorithm. Additionally, we provided the first theoretical evidence that probabilistic local updates significantly accelerate communication for ATC-based distributed algorithms.


\bibliographystyle{dcu}
\bibliography{autosam}


\end{document}